\documentclass{svmult}

\usepackage{makeidx}         
\usepackage{graphicx}        
\usepackage{multicol}        
\usepackage[bottom]{footmisc}
\usepackage{url}
\usepackage[english]{babel}
\renewcommand{\PackageWarningNoLine}[2]{}
\usepackage{amsmath}
\usepackage{amsfonts}

\begin{document}

\title*{Dirichlet-Neumann and Neumann-Neumann Waveform Relaxation for the Wave Equation}
\titlerunning{DNWR and NNWR Methods}
\author{Martin J. Gander\inst{1}, Felix Kwok\inst{1} and Bankim C. Mandal \inst{1}}
\institute{Department of Mathematics, University of Geneva, Switzerland. \texttt{Martin.Gander@unige.ch,Felix.Kwok@unige.ch,Bankim.Mandal@unige.ch}}
%

%
\maketitle
\abstract*{ We present a Waveform Relaxation (WR) version of the
  Dirichlet-Neumann and Neumann-Neumann algorithms for the wave
  equation in space time.  Each method is based on a non-overlapping
  spatial domain decomposition, and the iteration involves subdomain
  solves in space time with corresponding interface condition,
  followed by a correction step. Using a Laplace transform argument,
  for a particular relaxation parameter, we prove convergence of both
  algorithms in a finite number of steps for finite time
  intervals. The number of steps depends on the size of the subdomains
  and the time window length on which the algorithms are employed. We
  illustrate the performance of the algorithms with numerical results,
  and also show a comparison with classical and optimized Schwarz WR
  methods.}
\section{Introduction}
We present two new types of Waveform Relaxation (WR) methods for
hyperbolic problems based on the Dirichlet-Neumann and Neumann-Neumann
algorithms, and present convergence results for these methods. The
Dirichlet-Neumann algorithm for elliptic problems was first considered
by Bj{\o}rstad \& Widlund \cite{BjWid}; 
the Neumann-Neumann algorithm was introduced by Bourgat et
al. \cite{BouGT}. 
The performance of these algorithms for elliptic problems is now well
understood, see for example the book \cite{TosWid}.

To solve time-dependent problems in parallel, one can either
discretize in time to obtain a sequence of steady problems, and then
apply domain decomposition algorithms to solve the steady problems at
each time step in parallel, or one can first discretize in space and
then apply WR to the large system of ordinary differential equations
(ODEs) obtained from the spatial discretization. WR has its roots in
the work of Picard and Lindel\"{o}f, who
studied existence and uniqueness of solutions of ODEs in the late 19th
century. Lelarasmee, Ruehli and Sangiovanni-Vincentelli \cite{LelRue}
rediscovered WR as a parallel method for the solution of ODEs. The
main computational advantage of WR is parallelization, and the
possible use of different discretizations in different space-time
subdomains.

Domain decomposition methods for elliptic PDEs can be extended to
time-dependent problems by using the same decomposition in space.
This leads to WR type methods, see \cite{Bjor}.  The
systematic extension of the classical Schwarz method to time-dependent
parabolic problems was started independently in
\cite{GanStu,GilKel}. Like WR algorithms in general, the so-called
Schwarz Waveform Relaxation algorithms (SWR) converge relatively
slowly, except if the time window size is short. A remedy is to use
optimized transmission conditions, which leads to much faster
algorithms, see \cite{GH1} for parabolic problems, and
\cite{GHN} for hyperbolic problems. More recently, we studied the
WR extension of the Dirichlet-Neumann and Neumann-Neumann methods for
parabolic problems \cite{GKM1,Mandal,Kwok}. We proved for the heat
equation that on finite time intervals, the Dirichlet-Neumann Waveform
Relaxation (DNWR) and the Neumann-Neumann Waveform Relaxation (NNWR)
methods converge superlinearly for an optimal choice of the relaxation
parameter. DNWR and NNWR also converge faster than classical and
optimized SWR in this case. 

In this paper, we define DNWR and NNWR for
the second order wave equation
\begin{align}\label{mandalb_dd22:model}
\partial_{tt}u-c^{2}\Delta u & =  f(\boldsymbol{x},t), & &\boldsymbol{x}\in\Omega,0<t<T,\nonumber\\
u(\boldsymbol{x},0) &=  u_{0}(\boldsymbol{x}),\;  u_{t}(\boldsymbol{x},0)  =  v_{0}(\boldsymbol{x}), && \boldsymbol{x}\in\Omega,\\
u(\boldsymbol{x},t) & =  g(\boldsymbol{x},t), & & \boldsymbol{x}\in\partial\Omega,0<t<T,\nonumber
\end{align}
where $\Omega\subset\mathbb{R}^{d}$, $d=1,2,3$, is a bounded domain
with a smooth boundary, and $c$ denotes the wave speed, and we analyze the convergence of both algorithms for the 1d wave equation.


\section{Domain decomposition and algorithms}

To explain the new algorithms, we assume for simplicity that the
spatial domain $\Omega$ is partitioned into two non-overlapping
subdomains $\Omega_{1}$ and $\Omega_{2}$. We denote by $u_{i}$ the
restriction of the solution $u$ of (\ref{mandalb_dd22:model}) to
$\Omega_{i}$, $i=1,2$, and by $n_{i}$ the unit outward normal for
$\Omega_{i}$ on the interface
$\Gamma:=\partial\Omega_{1}\cap\partial\Omega_{2}$.

The {\em Dirichlet-Neumann Waveform Relaxation algorithm (DNWR)}
consists of the following steps: given an initial guess
$h^{0}(x,t)$, $t\in(0,T)$ along the interface $\Gamma\times(0,T)$, compute
for $k=1,2,...$ with $u_{1}^{k} = g,$ on $\partial\Omega_{1}\setminus\Gamma$ and $u_{2}^{k} = g,$ on $\partial\Omega_{2}\setminus\Gamma$ the approximations
\begin{equation}\label{mandalb_dd22:DNWR1} 
\arraycolsep0.2em 
\begin{array}{c}
\begin{array}{rcll}
\partial_{tt}u_{1}^{k}-c^{2}\Delta u_{1}^{k} & = & f, & \textrm{in}\; \Omega_{1},\\
u_{1}^{k}(\boldsymbol{x},0) & = & u_{0}(\boldsymbol{x}), & \textrm{in}\; \Omega_{1},\\
\partial_{t}u_{1}^{k}(\boldsymbol{x},0) & = & v_{0}(\boldsymbol{x}), &  \textrm{in}\; \Omega_{1},\\
u_{1}^{k} & = & h^{k-1}, & \textrm{on}\; \Gamma,
\end{array}\;
\begin{array}{rcll}
\partial_{tt}u_{2}^{k}-c^{2}\Delta u_{2}^{k} & = & f, & \textrm{in}\; \Omega_{2},\\
u_{2}^{k}(\boldsymbol{x},0) & = & u_{0}(\boldsymbol{x}), & \textrm{in}\; \Omega_{2},\\
\partial_{t}u_{2}^{k}(\boldsymbol{x},0) & = & v_{0}(\boldsymbol{x}),  &\textrm{in}\; \Omega_{2},\\
\partial_{n_{2}} u_{2}^{k} & = & -\partial_{n_{1}} u_{1}^{k}, & \textrm{on}\; \Gamma,\\
\end{array}\medskip \\
h^{k}(x,t)=\theta u_{2}^{k}\left|_{\Gamma\times(0,T)}\right.+(1-\theta)h^{k-1}(x,t),
\end{array}
\end{equation}
where $\theta\in(0,1]$ is a relaxation parameter. 

The {\em Neumann-Neumann Waveform Relaxation algorithm (NNWR)} starts
with an initial guess $w^{0}(x,t)$, $t\in(0,T)$ along the interface
$\Gamma\times(0,T)$ and then computes for $\theta\in(0,1]$ simultaneously for $i=1,2$ with $k=1,2,...$
\begin{equation}\label{mandalb_dd22:NNWR1} 
\arraycolsep0.01em 
\begin{array}{c}
\begin{array}{rcll}
\partial_{tt}u_{i}^{k}-c^{2}\Delta u_{i}^{k} & = & f, & \textrm{in}\; \Omega_{i},\\
u_{i}^{k}(\boldsymbol{x},0) & = & u_{0}(\boldsymbol{x}), & \textrm{in}\; \Omega_{i},\\
\partial_{t}u_{i}^{k}(\boldsymbol{x},0) & = & v_{0}(\boldsymbol{x}), & \textrm{in}\; \Omega_{i},\\
u_{i}^{k} & = & g, & \textrm{on}\; \partial\Omega_{i}\setminus\Gamma,\\
u_{i}^{k} & = & w^{k-1}, & \textrm{on}\; \Gamma,
\end{array} \begin{array}{rcll}
\partial_{tt}\psi_{i}^{k}-c^{2}\Delta\psi_{i}^{k} & = & 0, & \textrm{in}\; \Omega_{i},\\
\psi_{i}^{k}(\boldsymbol{x},0) & = & 0, & \textrm{in}\; \Omega_{i},\\
\partial_{t}\psi_{i}^{k}(\boldsymbol{x},0) & = & 0, & \textrm{in}\; \Omega_{i},\\
\psi_{i}^{k} & = & 0, & \textrm{on}\; \partial\Omega_{i}\setminus\Gamma,\\
\partial_{n_{i}}\psi_{i}^{k} & = & \partial_{n_{1}}u_{1}^{k}+\partial_{n_{2}}u_{2}^{k}, & \textrm{on}\; \Gamma,
\end{array}\medskip\\ 
w^{k}(x,t)=w^{k-1}(x,t)-\theta\left[\psi_{1}^{k}\left|_{\Gamma\times(0,T)}\right.+\psi_{2}^{k}\left|_{\Gamma\times(0,T)}\right.\right].
\end{array}
\end{equation}
\section{Kernel estimates and convergence analysis }

We present the case $d=1$, with $\Omega=(-a,b)$,
$\Omega_{1}=(-a,0)$ and $\Omega_{2}=(0,b)$.  By linearity, it suffices
to study the error equations, $f(\boldsymbol{x},t)=0$,
$g(\boldsymbol{x},t)=0$,
$u_{0}(\boldsymbol{x})=v_{0}(\boldsymbol{x})=0$ in
(\ref{mandalb_dd22:DNWR1}) and (\ref{mandalb_dd22:NNWR1}), and to
examine convergence to zero.

Our convergence analysis is based on Laplace transforms.  
The Laplace transform of a function $u(x,t)$ with respect to time $t$ is defined
by $\hat{u}(x,s)=\mathcal{L}\left\{ u(x,t)\right\}
:=\int_{0}^{\infty}e^{-st}u(x,t)\, dt$, $s \in {\mathbb C}$.  
Applying a Laplace transform to the DNWR algorithm in
(\ref{mandalb_dd22:DNWR1}) in 1d, we obtain for the transformed error equations
\begin{equation}\label{mandalb_dd22:DNWRL1}\arraycolsep0.0005em 
\begin{array}{c}
\begin{array}{rcll}   
(s^2\!-\!c^2\partial_{xx})\hat{u}_{1}^{k}(x,s)&=&0 &\!\!\! \!\!\textrm{in $(-a,0)$},\\   
\hat{u}_{1}^{k}(-a,s)&=&0,\\   
\hat{u}_{1}^{k}(0,s)&=&\hat{h}^{k-1}(s), 
\end{array}\quad 
\begin{array}{rcll}   
(s^2\!-\!c^2\partial_{xx})\hat{u}_{2}^{k}(x,s)&=&0 & \!\!\!\!\! \textrm{in $(0,b)$},\\  
\partial_{x}\hat{u}_{2}^{k}(0,s)&=&\partial_{x}\hat{u}_{1}^{k}(0,s),\\   
\hat{u}_{2}^{k}(b,s)&=&0, 
\end{array}\medskip\\
\hat{h}^{k}(s)=\theta\hat{u}_{2}^{k}(0,s)+(1-\theta)\hat{h}^{k-1}(s).
\end{array}
\end{equation}
Solving the two-point boundary value problems in
(\ref{mandalb_dd22:DNWRL1}), we get
$$\begin{array}{rclrcl} 
\hat{u}_{1}^{k}&=&\frac{\hat{h}^{k-1}(s)}{\sinh\left(as/c\right)}\sinh\left( (x+a)\frac{s}{c}\right), \quad&\hat{u}_{2}^{k}&=&\hat{h}^{k-1}(s)\frac{\coth(as/c)}{\cosh(bs/c)}\sinh\left((x-b)\frac{s}{c}\right), 
\end{array}$$ 
and inserting them into the updating condition (last line in
(\ref{mandalb_dd22:DNWRL1})), we get by induction
\begin{equation}\label{mandalb_dd22:DNiter}
\hat{h}^{k}(s)=\left[ 1-\theta-\theta\coth(as/c)\tanh(bs/c)\right] ^{k}\hat{h}^{0}(s),\quad k=1,2,\ldots
\end{equation}
Similarly, the Laplace transform of the NNWR algorithm in
(\ref{mandalb_dd22:NNWR1}) for the error equations yields for the
subdomain solutions
$$\begin{array}{rclrcl}   
\hat{u}_{1}^{k}(x,s)&=&\frac{\hat{w}^{k-1}(s)}{\sinh(as/c)}\sinh\left( (x+a)\frac{s}{c}\right),\; &\hat{u}_{2}^{k}(x,s)&=&-\frac{\hat{w}^{k-1}(s)}{\sinh(bs/c)}\sinh\left( (x-b)\frac{s}{c}\right),\\   
\hat{\psi}_{1}^{k}(x,s)&=&\frac{\hat{w}^{k-1}(s)\Psi(s)}{\cosh(as/c)}\sinh\left( (x+a)\frac{s}{c}\right), \; &\hat{\psi}_{2}^{k}(x,s)&=&\frac{\hat{w}^{k-1}(s)\Psi(s)}{\cosh(bs/c)}\sinh\left( (x-b)\frac{s}{c}\right),
\end{array}$$ 
where $\Psi(s)=\left[ \coth(as/c)+\coth(bs/c)\right] $. Therefore, in Laplace space the updating condition in (\ref{mandalb_dd22:NNWR1}) becomes 
\begin{equation}\label{mandalb_dd22:NNiter}
\hat{w}^{k}(s)=\left[1-\theta\left( 2+\frac{\coth(as/c)}{\coth(bs/c)}+\frac{\coth(bs/c)}{\coth(as/c)}\right) \right]^{k}\hat{w}^{0}(s),\quad k=1,2,\ldots
\end{equation}

\begin{theorem}[Convergence, symmetric decomposition] 
For a symmetric decomposition, $a=b$, convergence is linear for the
DNWR (\ref{mandalb_dd22:DNWR1}) with $\theta\in(0,1)$,
$\theta\ne\frac{1}{2}$, and for the NNWR (\ref{mandalb_dd22:NNWR1}) with
$\theta\in(0,1)$, $\theta\ne\frac{1}{4}$. If
$\theta=\frac{1}{2}$ for DNWR, or $\theta=\frac{1}{4}$ for NNWR,
convergence is achieved in two iterations.
\end{theorem}
\begin{proof}
For $a=b$, equation (\ref{mandalb_dd22:DNiter}) reduces to
$\hat{h}^{k}(s)=(1-2\theta)^{k}\hat{h}^{0}(s),$ which has the simple
back transform $h^{k}(t)=(1-2\theta)^{k}h^{0}(t)$.  Thus for the DNWR
method, the convergence is linear for $0<\theta<1,\theta\neq\frac{1}{2}$.  For
$\theta=\frac{1}{2}$, we have $h^{1}(t)=0$. Hence, one more iteration produces
the desired solution on the whole domain.

For the NNWR algorithm, inserting $a=b$ into equation
(\ref{mandalb_dd22:NNiter}), we obtain similarly
$w^{k}(t)=(1-4\theta)^{k}w^{0}(t)$, which leads to the second result.
$\hfill\qed$
\end{proof}
We next analyze the case of an asymmetric decomposition, $a\neq b$. 
\begin{lemma}\label{mandalb_dd22:Lemma1}
Let $a,b>0$ and $s\in{\mathbb C}$, with $\textrm{Re}(s)>0$. Then, we
have the identity
\begin{multline*}
G_{b}^{a}(s):=\coth(as/c)\tanh(bs/c)-1\\
=2{\displaystyle \sum_{m=1}^{\infty}}e^{-2ams/c}-2{\displaystyle \sum_{n=1}^{\infty}}(-1)^{n-1}e^{-2bns/c}
-4{\displaystyle \sum_{n=1}^{\infty}}{\displaystyle \sum_{m=1}^{\infty}}(-1)^{n-1}e^{-2(bn+am)s/c}.
\end{multline*}
\end{lemma}
\begin{proof}
Using that $\left|e^{-2bs/c}\right|<1$ for $\textrm{Re}(s)>0$, we
expand $\left(1+e^{-2bs/c}\right)^{-1}$ into an infinite binomial
series to obtain
\begin{equation*}
\tanh\left(\frac{bs}{c}\right) \!= \frac{e^{\frac{bs}{c}}-e^{-\frac{bs}{c}}}{e^{\frac{bs}{c}}+e^{-\frac{bs}{c}}}=\!\left(1-e^{-\frac{2bs}{c}}\right)\!\left(1+e^{-\frac{2bs}{c}}\right)^{-1} \!\!\! =\! 1-2{\displaystyle \sum_{n=1}^{\infty}}(-1)^{n-1}e^{-\frac{2bns}{c}}.
\end{equation*}
Similarly, we get $\textrm{coth}(as/c)=1+2{\displaystyle
  \sum_{m=1}^{\infty}}e^{-\frac{2ams}{c}}$, and multiplying the two
and subtracting $1$, we obtain the expression for $G_b^a(s)$ in the
Lemma.$\hfill\qed$
\end{proof}

Using $G_{b}^{a}(s)$ from Lemma \ref{mandalb_dd22:Lemma1}, we obtain for
(\ref{mandalb_dd22:DNiter})
\begin{equation}\label{mandalb_dd22:DNiterfinal}
\hat{h}^{k}(s)=\left\{ \left(1-2\theta\right)-\theta G_{b}^{a}(s)\right\} ^{k}\hat{h}^{0}(s).
\end{equation}
Now if $\theta=\frac{1}{2}$, we see that the linear factor in
(\ref{mandalb_dd22:DNiterfinal}) vanishes, and convergence will be
governed by convolutions of $G_{b}^{a}(s)$. We show next that this
choice also gives finite step convergence, but the number of steps
depends on the length of the time window $T$.
\begin{theorem}[Convergence of DNWR, asymmetric decomposition] 
Let $\theta=\frac{1}{2}$. Then the DNWR algorithm converges in at most $k+1$
iterations for two subdomains of lengths $a\ne b$, if the time window
length $T$ satisfies $T/k\leq2\min\left\{ a/c,b/c\right\}$, where $c$
is the wave speed.
\end{theorem}
\begin{proof}
With $\theta=\frac{1}{2}$ we obtain from
(\ref{mandalb_dd22:DNiterfinal}) for $k=1,2,\ldots$
\begin{multline}\label{mandalb_dd22:calcul}
\hat{h}^{k}(s)  =  \left(-\frac{1}{2}\right)^{k}\left\{ G_{b}^{a}(s)\right\} ^{k}\hat{h}^{0}(s)
  =  \left[ -e^{-\frac{2as}{c}}+e^{-\frac{2bs}{c}}+\left({\displaystyle \sum_{n>1}^{\infty}}(-1)^{n-1}e^{-\frac{2bns}{c}}\right.\right.\\
  -\left.\left. {\displaystyle \sum_{m>1}^{\infty}}e^{-\frac{2ams}{c}}+2{\displaystyle \sum_{m=1}^{\infty}}{\displaystyle \sum_{n=1}^{\infty}}(-1)^{n-1}e^{-\frac{2(am+bn)s}{c}}\right)\right] ^{k}\hat{h}^{0}(s)
  =  (-1)^{k}e^{-\frac{2aks}{c}}\hat{h}^{0}(s)\\
  +  e^{-\frac{2bks}{c}}\hat{h}^{0}(s)+\!\left(\!{\displaystyle \sum_{l>k}^{\infty}}p_{l}^{(k)}e^{-\frac{2bls}{c}}\!+\!{\displaystyle \sum_{l>k}^{\infty}}q_{l}^{(k)}e^{-\frac{2als}{c}}\!+\!{\displaystyle \sum_{m+n\geq k}}r_{m,n}^{(k)}e^{-\frac{2(am+bn)s}{c}}\right)\hat{h}^{0}(s),
\end{multline}
$p_{l}^{(k)},q_{l}^{(k)},r_{m,n}^{(k)}$ being the corresponding coefficients. Using the inverse Laplace transform 
\begin{equation}\label{mandalb_dd22:Lformula}
\mathcal{L}^{-1}\left\{ e^{-\alpha s}\hat{g}(s)\right\} =H(t-\alpha)g(t-\alpha),
\end{equation}
$H(t)$ being Heaviside step function, we obtain
\begin{multline*}
h^{k}(t)=(-1)^{k}h^{0}(t-2ak/c)H(t-2ak/c)+h^{0}(t-2bk/c)H(t-2bk/c)\\
+ {\displaystyle \sum_{l>k}^{\infty}}p_{l}^{(k)}h^{0}(t-2bl/c)H(t-2bl/c) +
{\displaystyle \sum_{l>k}^{\infty}}q_{l}^{(k)}h^{0}(t-2al/c)H(t-2al/c)\\
+{\displaystyle \sum_{m+n\geq k}}r_{m,n}^{(k)}h^{0}(t-2(am+bn)/c)H(t-2(am+bn)/c).
\end{multline*}
Now if we choose our time window such that $T\leq2k\min\left\{ \frac{a}{c},\frac{b}{c}\right\} $,
then $h^{k}(t)=0$, and therefore one more iteration produces
the desired solution on the entire domain.$\hfill\qed$
\end{proof}
Using $G_{b}^{a}(s)$ from Lemma \ref{mandalb_dd22:Lemma1}, we can
also rewrite (\ref{mandalb_dd22:NNiter}) in the form
\begin{equation}\label{mandalb_dd22:NNiterfinal}
\hat{w}^{k}(s)=\left\{ \left(1-4\theta\right)-\theta\left(G_{b}^{a}(s)+G_{a}^{b}(s)\right)\right\} ^{k}\hat{w}^{0}(s),\quad k=1,2,\ldots,
\end{equation}
and we see that for NNWR, the choice $\theta=\frac{1}{4}$ removes the linear
factor.
\begin{theorem}[Convergence of NNWR, asymmetric decomposition] 
  Let $\theta=\frac{1}{4}$. Then the NNWR algorithm converges in at
  most $k+1$ iterations for two subdomains of lengths $a\ne b$, if the
  time window length $T$ satisfies $T/k\leq4\min\left\{
  a/c,b/c\right\}$, $c$ being again the wave speed.
\end{theorem}
\begin{proof}
With $\theta=\frac{1}{4}$ we obtain from
(\ref{mandalb_dd22:NNiterfinal}) with a similar calculation as in
(\ref{mandalb_dd22:calcul})
\begin{multline*}
\hat{w}^{k}(s)  =  \left(-\frac{1}{4}\right)^{k}\left[ G_{b}^{a}(s)+G_{a}^{b}(s)\right] ^{k}\hat{w}^{0}(s)
  =  \left[ -{\displaystyle \sum_{m=1}^{\infty}}\left(e^{-\frac{4ams}{c}}+e^{-\frac{4bms}{c}}\right)\right.\\
  \!+\! \left.{\displaystyle \sum_{m=1}^{\infty}}{\displaystyle \sum_{n=1}^{\infty}}(-1)^{n-1}\left(e^{-\frac{2(am+bn)s}{c}}\!+\! e^{-\frac{2(an+bm)s}{c}}\right)\right] ^{k}\!\!\hat{w}^{0}(s)
  = (-1)^{k}e^{-\frac{4aks}{c}}\!\hat{w}^{0}(s)\\
  +\left[\!(-1)^{k}e^{-\frac{4bks}{c}}\!+\!\left(\!{\displaystyle \sum_{l>k}^{\infty}}d_{l}^{(k)}e^{-\frac{4als}{c}}\!\!+\!\!{\displaystyle \sum_{l>k}^{\infty}}z_{l}^{(k)}e^{-\frac{4bls}{c}}\!\!+\!\!\!\!{\displaystyle \sum_{m+n\geq2k}}\!\!\!j_{m,n}^{(k)}e^{-\frac{2(am+bn)s}{c}}\!\right)\!\right]\!\hat{w}^{0}(s),
\end{multline*}
where $d_{l}^{(k)},z_{l}^{(k)},j_{m,n}^{(k)}$ are the corresponding coefficients.
Now we use (\ref{mandalb_dd22:Lformula}) to back transform and obtain 
\begin{multline*}
w^{k}(t)=(-1)^{k}w^{0}(t-4ak/c)H(t-4ak/c)+(-1)^{k}w^{0}(t-4bk/c)H(t-4bk/c)\\
+ {\displaystyle \sum_{l>k}^{\infty}}d_{l}^{(k)}w^{0}(t-4al/c)H(t-4al/c)
+ {\displaystyle \sum_{l>k}^{\infty}}z_{l}^{(k)}w^{0}(t-4bl/c)H(t-4bl/c)\\
+{\displaystyle \sum_{m+n\geq2k}}j_{m,n}^{(k)}w^{0}(t-2(am+bn)/c)H(t-2(am+bn)/c).
\end{multline*}
So for $T\leq4k\min\left\{ \frac{a}{c},\frac{b}{c}\right\} $, we
get $w^{k}(t)=0$, and the conclusion follows.$\hfill\qed$
\end{proof}

\section{Numerical Experiments}

We perform now numerical experiments to measure the
actual convergence rate of the discretized DNWR and NNWR algorithms for the model problem
\begin{align}\label{mandalb_dd22:Numermodel}
\partial_{tt}u-\partial_{xx}u&=0, && x\in(-3,2),t>0,\nonumber\\
u(x,0)&=0,\: u_{t}(x,0)=xe^{-x},&& -3<x<2,\\
u(-3,t)&=-3e^3t,\: u(2,t)=2e^{-2}t, && t>0,\nonumber
\end{align}
with $\Omega_{1}=(-3,0)$ and $\Omega_{2}=(0,2)$, so that $a=3$ and
$b=2$ in (\ref{mandalb_dd22:DNWRL1}, \ref{mandalb_dd22:DNiter},
\ref{mandalb_dd22:NNiter}).  We discretize the equation using the
centered finite difference in both space and time (Leapfrog scheme) on a grid with $\Delta x=\Delta t=2{\times}10^{-2}$. The error is calculated by $\|u-u^k\|_{L^\infty(0,T;L^2(\Omega))}$, where $u$ is the discrete monodomain solution and $u^k$ is the discrete solution in $k$-th iteration.

We test the DNWR algorithm by choosing $h^{0}(t)=t^{2},t\in(0,T]$ as
an initial guess. In Figure \ref{mandalb_dd22:Fig1}, 
\begin{figure}[t]
\centering
  \mbox{\includegraphics[width=0.51\textwidth]{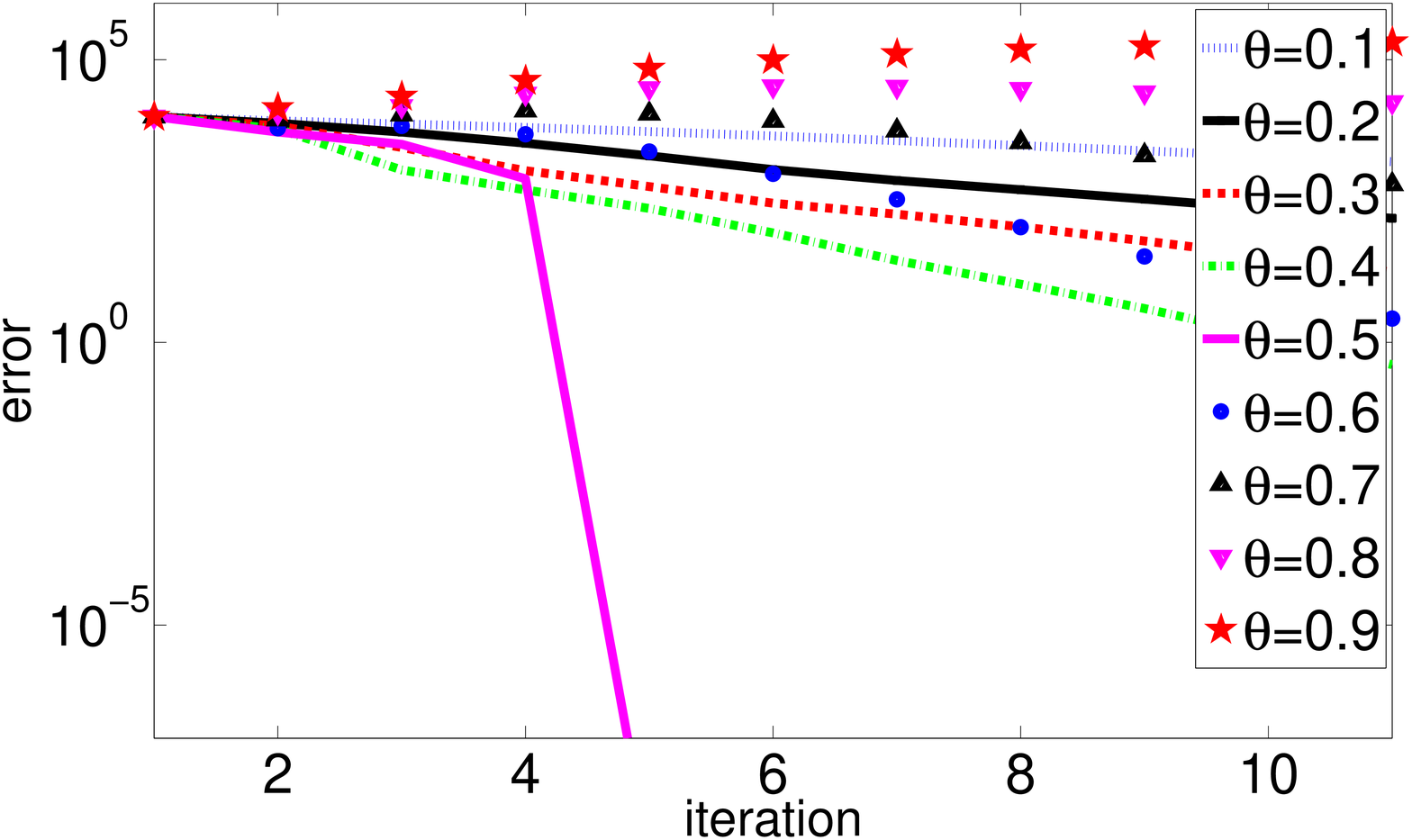}
  \includegraphics[width=0.51\textwidth]{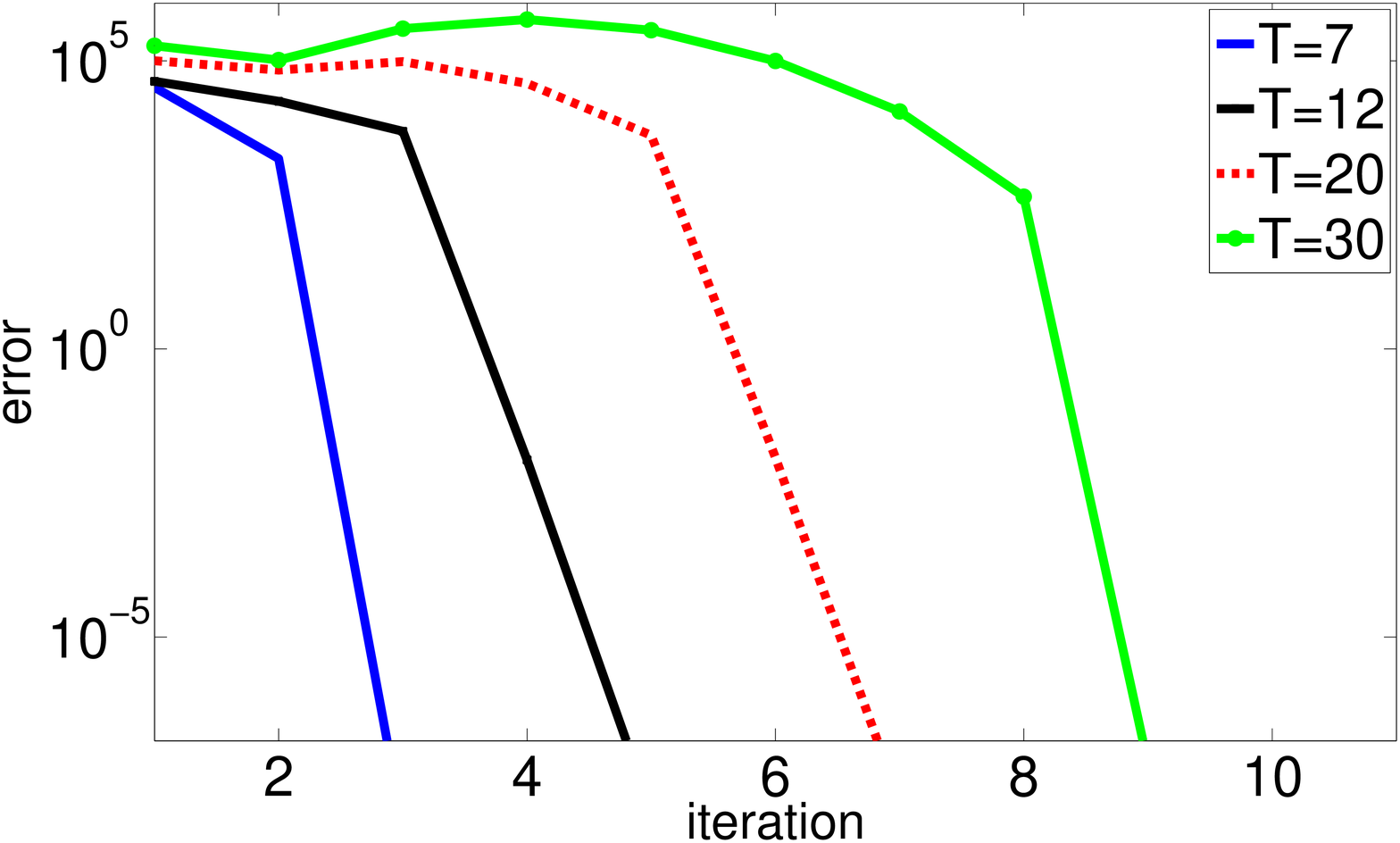}}
\caption{Convergence of DNWR for various values of $\theta$ and $T=16$ on
the left, and for various lengths $T$ of the time window and
 $\theta=\frac{1}{2}$ on the right}
\label{mandalb_dd22:Fig1}
\end{figure} 
we show the convergence behavior for different values of the parameter
$\theta$ for $T=16$ on the left, and on the right for the best
parameter $\theta=\frac{1}{2}$ for different time window length $T$. Note that for some values of $\theta$ ($>0.7$) we get divergence.
For the NNWR method, with the same initial guess,
we show in Figure \ref{mandalb_dd22:Fig2}
\begin{figure}[t]
\centering
  \mbox{\includegraphics[width=0.51\textwidth]{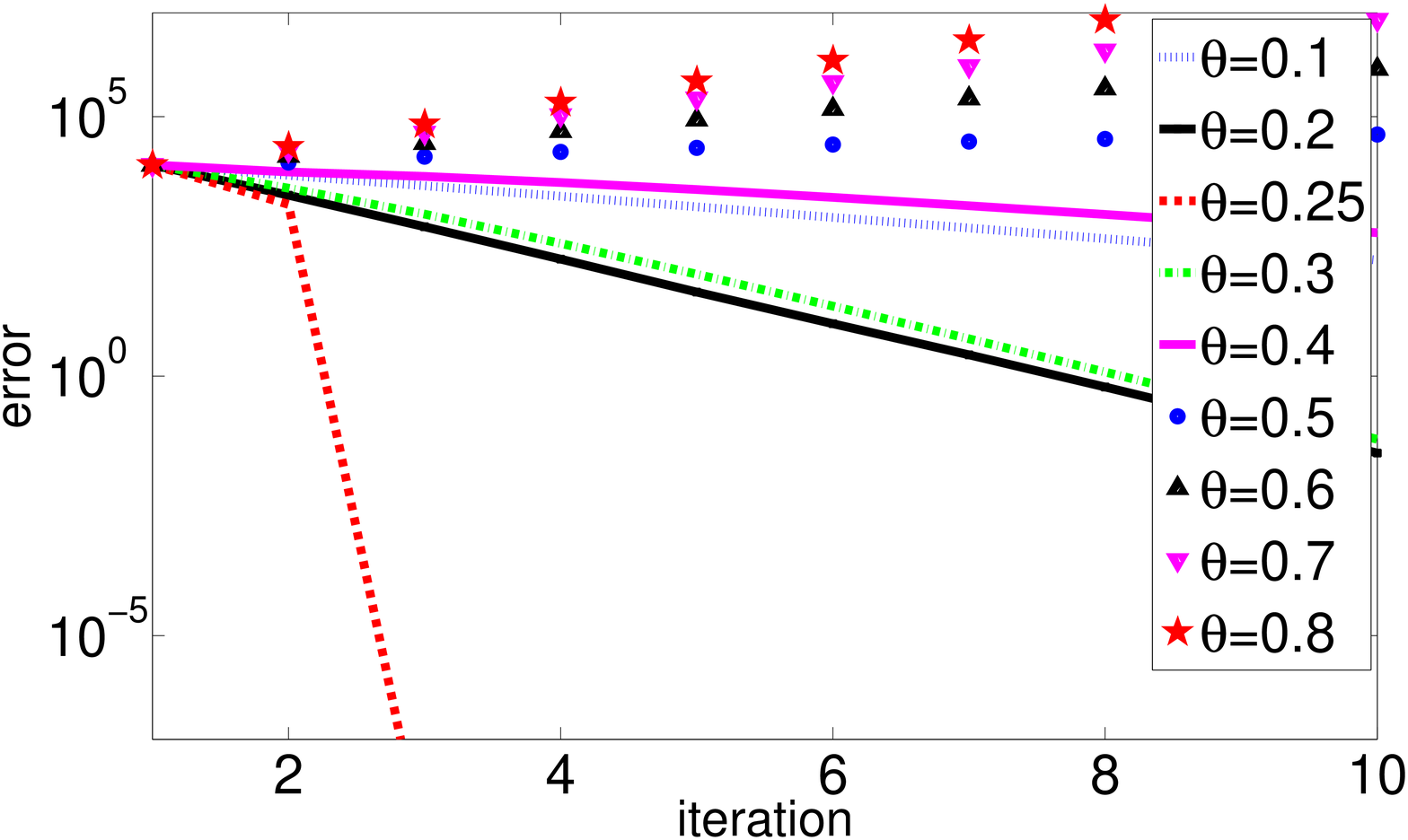}
  \includegraphics[width=0.51\textwidth]{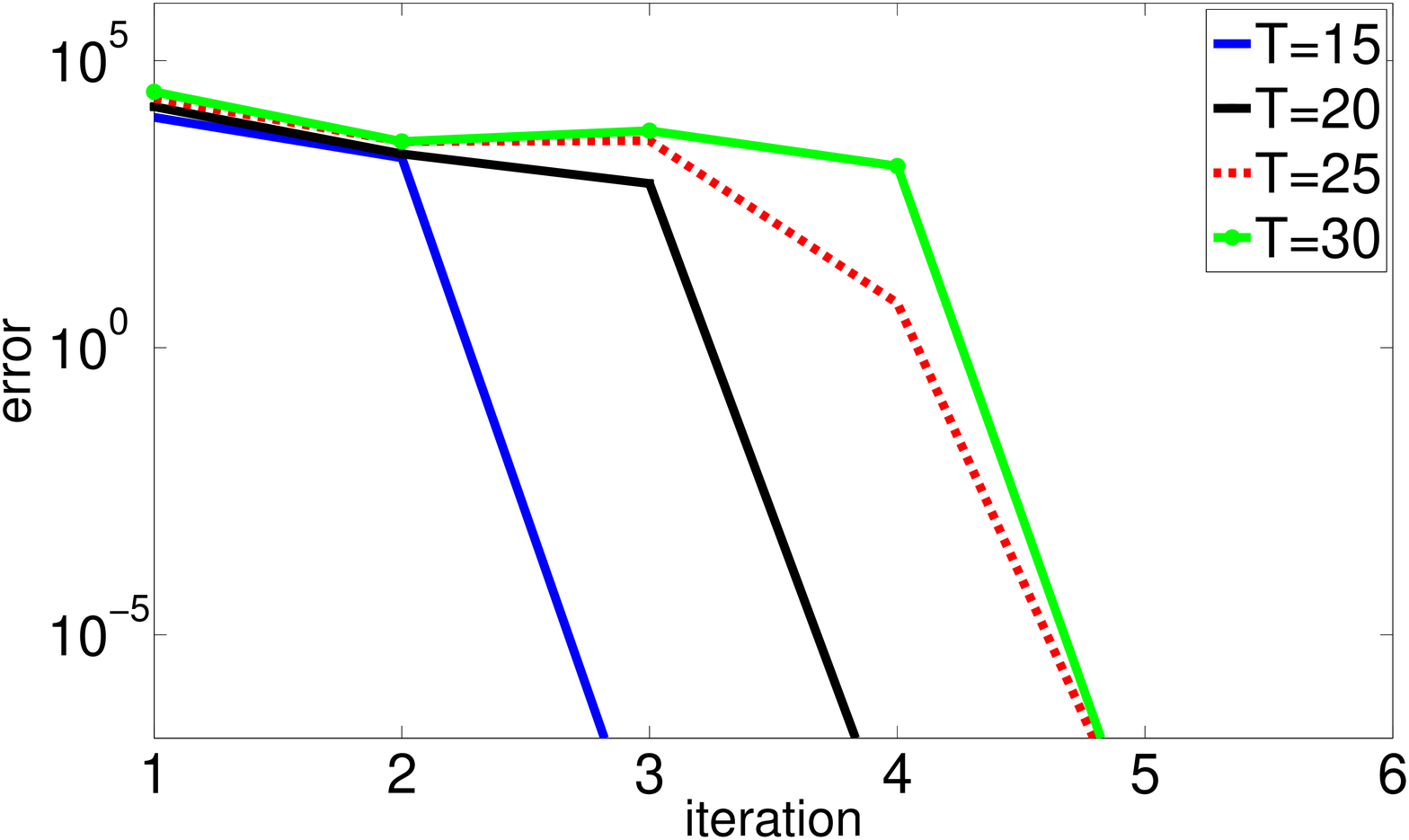}}
  \caption{Convergence of NNWR with various values of $\theta$ for $T=16$ on
the left, and for various lengths $T$ of the time window and
 $\theta=\frac{1}{4}$ on the right}
\label{mandalb_dd22:Fig2}
\end{figure} 
on the left the convergence curves for different values of $\theta$ for $T=16$, and on the right the results for the best parameter $\theta=\frac{1}{4}$ for different time window lengths
$T$.

We finally compare in Figure \ref{mandalb_dd22:Fig3} 
the performance of the DNWR and NNWR algorithms
with the Schwarz Waveform Relaxation (SWR) algorithms from
\cite{GHN} with and without overlap. 
\begin{figure}[t]
\centering
  \mbox{\includegraphics[width=0.51\textwidth]{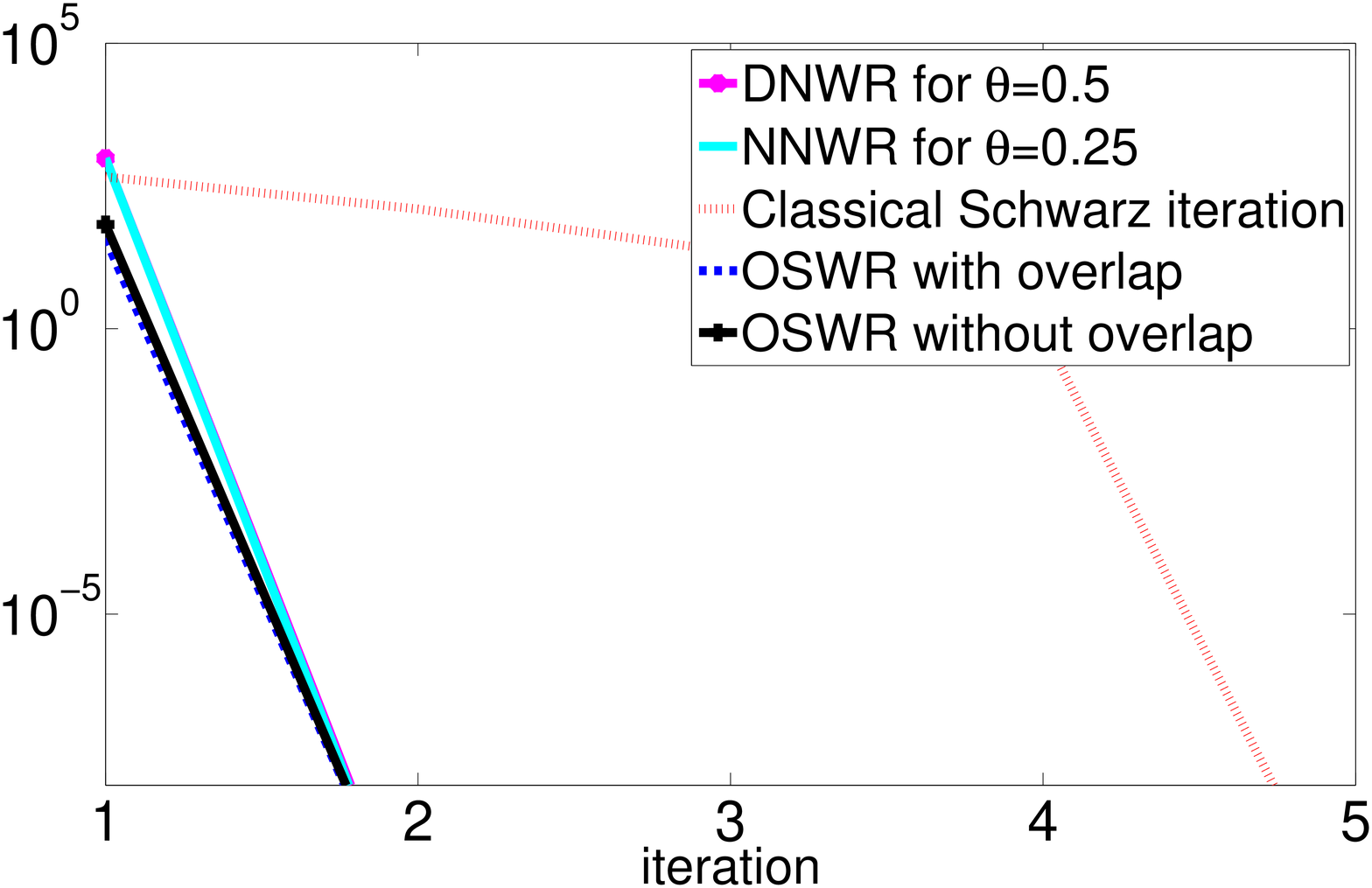}
  \includegraphics[width=0.51\textwidth]{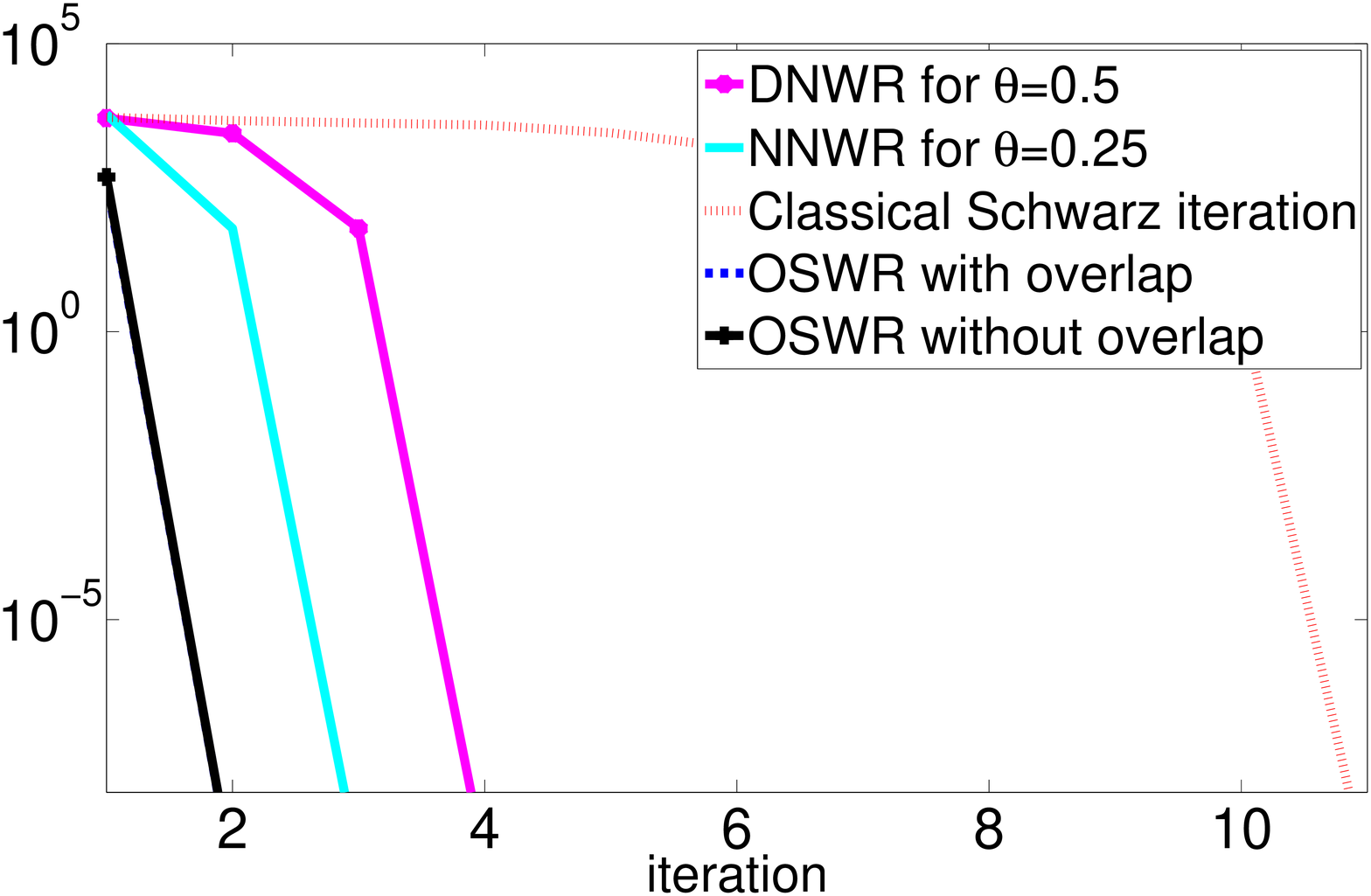}}
  \caption{Comparison of DNWR, NNWR, and SWR for $T=4$ on
the left, and $T=10$ on the right}
\label{mandalb_dd22:Fig3}
\end{figure} 
We consider the same model problem (\ref{mandalb_dd22:Numermodel}) with Dirichlet
boundary conditions along the physical boundary. We use for the overlapping Schwarz variant an overlap
of length $24\Delta x$, where $\Delta x=1/50$. We observe that the DNWR
and NNWR algorithms converge as fast as the Schwarz WR algorithms for
smaller time windows $T$. Due to the local nature of the Dirichlet-to-Neumann operator in 1d \cite{GHN}, SWR converges in a finite number of iterations just like DNWR
and NNWR. In higher dimensions, however, SWR will no longer converge
in a finite number of steps, but DNWR and NNWR will \cite{GKM2}.

\section{Conclusions }

We introduced the DNWR and NNWR algorithms for the second order wave equation, and analyzed
their convergence properties for the 1d case and a two
subdomain decomposition. We showed that for a particular choice of the
relaxation parameter, convergence can be achieved in a finite number
of steps. Choosing the time window lengh carefully, these algorithms
can be used to solve such problems in two iterations only. For a detailed analysis for the case of multiple
subdomains, see \cite{GKM2}.
\nocite{*}

\bibliographystyle{spmpsci}
\bibliography{mandalb_dd22} 

\end{document}